\documentclass[a4paper,notitlepage,twoside,reqno,12pt]{amsart}
\usepackage[all]{xy}
\usepackage{framed}
\usepackage{CJK}
\usepackage{mathrsfs}
\usepackage{bbm,pifont,latexsym}
\usepackage{anysize}  
\usepackage{enumerate}
\usepackage{fancyhdr} 
\usepackage{dcolumn,indentfirst,color}
\usepackage[pdfpagemode=UseNone,pdfstartview=FitH]{hyperref}
\usepackage{amsmath,amssymb,amscd,amsthm,amsfonts,mathrsfs}
\usepackage{color,graphicx,xcolor,graphics}
\usepackage{titlesec} 
\usepackage{titletoc} 

\marginsize{3cm}{3cm}{3cm}{3cm}
\titlecontents{section}[20pt]{\addvspace{2pt}\filright}
{\contentspush{\thecontentslabel. }}
{}{\titlerule*[8pt]{}\contentspage}
\newtheorem{thm}{Theorem}[section]
\newtheorem{lem}{Lemma}[section]

\newtheorem{pro}[thm]{Proposition}

\newtheorem*{thmA}{Theorem A}

\newcommand{\cbar}{\widehat{\mathbb{C}}}
\newcommand{\ra}{\rightarrow}
\newcommand{\pa}{\partial}
\newcommand{\sm}{\setminus}
\newcommand{\wh}{\widehat}
\newcommand{\ol}{\overline}
\newcommand{\mc}{\mathcal}

\makeatletter\@addtoreset{equation}{section}\makeatother 

\titleformat{\section}{\large}{\textbf{\thesection.}}{1em}{\textbf}
\titleformat{\subsection}{\normalsize}{\textbf{\thesubsection.}}{1em}{\textbf}
\titleformat{\subsubsection}{\normalsize}{\thesubsubsection.}{1em}{\textbf}

\begin{document}

\title{Non-recurrent rational maps with disconnected Julia set}
\author{Yan Gao}
\address{Yan Gao, School of Mathematical Sciences, Shenzhen University, Shenzhen 518061, China}
\email{gyan@szu.edu.cn}

\author{Lele Xu}
\address{Lele Xu, School of Mathematical Sciences, Shenzhen University, Shenzhen, 518052, P. R. China}
\email{2450191001@mails.szu.edu.cn}

\author{Luxian Yang}
\address{Luxian Yang, School of Mathematical Sciences, Shenzhen University, Shenzhen, 518052, P. R. China}
\email{lxyang@szu.edu.cn}
	

\begin{abstract}
We prove that every wandering exposed Julia component of a rational map is to a singleton, provided that each wandering Julia component containing critical points is non-recurrent. Moreover, we show that the Julia set contains only finitely many periodic complex-type components if each wandering Julia component containing critical values is non-recurrent.
\end{abstract}

\maketitle

\section{Introduction}\label{sec:intro}
Let $f:\widehat{\mathbb{C}}\rightarrow \widehat{\mathbb{C}}$ be a rational map of degree at least two. The {\bf Fatou set} $F_f$ consists of points $z\in\cbar$ for which the iteration sequence $\{f^n\}_{n =1}^{\infty}$ forms a normal family in a neighborhood of $z$. Its complement is the  {\bf Julia set} $J_f$.  The {\bf postcritical set} of $f$ is defined as 
$$P_f=\ol{\bigcup_{n>0, f'(c)=0} f^n(c)}.$$
For basic properties of Fatou and Julia sets, 
one can refer to \cite{beardon1991the,milnor2006dynamics}. A {central} issue in complex dynamics is to investigate the structures of the Julia sets.

The Julia set of $f$ is either connected or disconnected. In this paper we focus only on the disconnected case. 
We use  {\bf Julia component} to represent a component of the Julia set. Under iteration, a Julia component $K$ is either {\bf preperiodic}, i.e., $f^m(K)=f^{p+m}(K)$ for some $p\geq1,m\geq 0$; or {\bf wandering}. A preperiodic Julia component $K$ is  called {\bf periodic} if $m=0$ in the definition, and the smallest such $p$ is called the {\bf period} of $K$. 

In combinatorial terms, a compact connected set in $\wh{\mathbb{C}}$ is called to be of {\bf full type} if it has one complementary component,  of {\bf annuli type} if it has two complementary components, and of {\bf complex type} otherwise.

The structure of periodic Julia components was characterized by  McMullen \cite{mcmullen1988automorphisms}.  

\begin{thmA}\cite[Theorem 3.4]{mcmullen1988automorphisms}
Let $f$ be a rational map of degree at least two, and let $K$ be a non-singleton periodic Julia component of $f$  with period $p\ge 1$. Then there exists a rational map $g$ and a quasiconformal map $\phi:\cbar\ra\cbar$ such that
\begin{itemize}
	\item[(1)] $\phi(K)=J_g$, and
	\item[(2)] $\phi\circ f^p(z)=g\circ\phi(z)$ for $z\in K$.
\end{itemize}
\end{thmA}

From this theorem, one can easily deduce that any rational map with disconnected Julia set has countably many periodic Julia components. On the other hand, no examples are known of rational maps with infinitely many periodic complex-type Julia components.\footnote{In the family $f_\lambda(z)=z^n+\lambda/z^n,\lambda\in\mathbb{C}$, one can find rational maps with infinitely many periodic Julia component of full type or annuli type.} This naturally leads to the following problem:
\vspace{5pt}

\noindent\emph{Problem 1. Is there a rational map with infinitely many complex-type periodic Julia components?}\vspace{5pt}

Since a Julia set is either connected or has uncountablely many components \cite{beardon1991iteration,milnor2006dynamics}, any rational map with disconnected Julia set necessarily has uncountably many wandering Julia components. It is not difficult to check that Problem 1 has a negative answer for rational maps whose wandering Julia components avoid critical points, in particular, for \emph{geometrically finite} maps (i.e. $P_f\cap J_f$ is a finite set) \cite{pilgrim2000rational}. 

A basic combinatorial problem about wandering Julia components is the \emph{existence of complex-type wandering Julia components}.\footnote{In the family $f_\lambda(z)=z^n+\lambda/z^n,\lambda\in\mathbb{C}$, one can find rational maps with eventually full-type or annuli-type wandering Julia components} That is, does there exist a rational map with a  wandering Julia component such that each component in its orbit is of complex type?
This problem was solved in the negative for geometrically finite rational maps by Pilgrim-Tan \cite{pilgrim2000rational}, and for cubic rational maps by Cui-Peng-Yang \cite{cui2024wandering}.

Another important problem about wandering Julia components concerns the analytic property of full-type or annuli-type components.  Branner and Hubbard conjectured that every wandering Julia component of a polynomial, which is necessarily of full-type, is a singleton \cite{branner1992iteration}.  This conjecture was independently confirmed by Qiu-Yin \cite{qiu2009proof} and Kozlovski-van Strien \cite{kozlovski2009local}. Tan and Pilgrim \cite{pilgrim2000rational} further proved that for geometrically finite rational maps, every full-type wandering Julia component is a singleton, while every annuli-type wandering Julia component is a Jordan curve.

Motivated by these works, a Julia component is called {\bf exposed} if it intersects the boundary of a Fatou domain. Every wandering exposed Julia component is eventually of full-type (see Lemma \ref{lem:basic}). As a natural generalization of polynomial case, we may ask:\vspace{5pt}
 
\noindent\emph{Problem 2. Is every wandering exposed Julia component a singleton?}\vspace{5pt}

For non-polynomial rational maps,  these two problems have only been solved in the geometrically finite case (\cite{cui2011topology} for Problem 1 and \cite{pilgrim2000rational} for Problem 2). In this paper, we will answer the two problems for a class of geometrically infinite maps, usually called ``critical non-recurrent'' maps. 

A connected and compact set $E\subset \cbar$ is called {\bf non-recurrent} under $f$, if there exists an open neighborhood $D$ of $E$ such that $f^n(E)\cap D=\emptyset$ for all $n\geq1$. 

\begin{thm}\label{thm:main2}
Let $f:\widehat{\mathbb{C}}\rightarrow \widehat{\mathbb{C}}$ be a rational map with disconnected Julia set. Suppose that each wandering exposed Julia component of $f$ containing critical points is non-recurrent. Then every wandering exposed Julia component of $f$ is a singleton. 
\end{thm}

The basic tool in the proof of Theorem \ref{thm:main2} is the construction of puzzle pieces from a periodic Fatou domain, a method first introduced by Branner-Hubbard \cite{branner1992iteration} for polynomials. Unlike the polynomial case, however, some puzzle pieces are not mapped to puzzle pieces under a rational map. Our main work is to deal with this combinatorial complexity. Theorem \ref{thm:main2} will be proved in Section 2. \vspace{3pt}

Under a stronger assumption, we obtain an answer to Problem 1.

\begin{thm}\label{thm:main1}
Let $f:\widehat{\mathbb{C}}\rightarrow \widehat{\mathbb{C}}$ be a rational map with disconnected Julia set. Suppose that every wandering Julia component containing critical values is non-recurrent. Then $J_f$ has finitely many periodic components of complex type. 
\end{thm}

The idea of the proof is similar to that in \cite{cui2024wandering}. The main different lies that,  the authors of \cite{cui2024wandering} study the configuration of complex-type Julia components within a single wandering orbit, while here we instead consider the configuration of all complex-type periodic Julia components. Theorem \ref{thm:main1} will be proved in Section 3.

\vspace{0.5cm}
\textbf{Acknowledgment.} This work is supported by the National Key R\&D Program of China (Grant no.\,2021YFA1003203),
 the NSFC (Grant nos.\,12131016 and \,12322104), and the NSFGD (Grant no.\,2023A1515010058).

\section{Wandering exposed  components}
In this section, we will prove Theorem \ref{thm:main2}. We always assume that $f:\wh{\mathbb{C}}\rightarrow \wh{\mathbb{C}}$ is a rational map of degree $d\geq 2$, and $U$ is an $f$-fixed multiply-connected Fatou domain. If $U$ is a parabolic basin, let $x_*\in\partial U$ such that $U$ is the immediate parabolic basin of the parabolic fixed point $x_*$. 

\subsection{Classification of  components of $\partial U$. }

Let $E\subset \mathbb{C}$ be a connected and compact set.  We define the {\bf filling} $\wh{E}$ of $E$ as the union of $E$ together with all bounded components of $\mathbb{C}\setminus E$. Note that $E$ is of full type if and only if $E=\wh{E}$.

Let $\mc{E}$ denote the collection of connected components of $\partial U$. Without loss of generality, we assume that $\infty\in U$.  Then $\wh{E}\cap U=\emptyset$ for any  $E\in\mc{E}$.  

An element $E\in \mc{E}$ is called  {\bf critical} if $\wh{E}$ contains critical points of $f$, and is called \textbf{essential} if $\wh{E}\cap f^{-1}(U)\ne \emptyset$. The following facts are immediate:
\begin{enumerate}
\item there are at most $d-1$ essential elements in $\mc{E}$;\vspace{2pt}

\item if $E\in\mc{E}$, then $f(E)\in \mc{E}$;
\item if $E\in\mc{E}$ is non-essential, then $f(\wh{E})=\wh{f(E)}$.
\end{enumerate}
Fact (2) above allow us to define a self-map 
$$
\sigma_f:\mc{E} \to \mc{E}, \ E\mapsto f(E)
$$ 
This map is surjective since $f(U)=U$. 
 
 Using the map $\sigma_f$, we can naturally define the  {\bf (pre)periodic} or {\bf wandering} elements of $\mc{E}$. Moreover, for any two elements $E,E'\in \mc{E}$, we say that the orbit of $E$ (under $\sigma_f$)  {\bf accumulates} to $E'$ if for any neighborhood $D$ of $E'$, there exists an integer $n_{D}\geq 1$ such that $\sigma_f^{n_D}(E)\cap D\neq\emptyset$. 

\begin{lem}\label{lem:basic}
For any $E\in\mc{E}$, the following statements hold. 
\begin{enumerate}
\item If $E$ is full, then $f(E)$ is full.
\item If $E$ is wandering, then $\sigma_f^{n}(E)$ is full for every sufficiently large $n$.
\item  If $E$ is essential, then it is a non-full and critical element of $\mc{E}$.
\item If $E$ is non-full while $\sigma_f(E)$ is, then $E$ is essential and contains critical points of $f$.
\end{enumerate}
\end{lem}
\begin{proof}
(1) Assume that $f(E)$ is not full. Then there exists a finitely connected set $A$ containing $f(E)$ such that $A\sm f(E)$ contains no critical values. Let $A_1$ be the component of $f^{-1}(A)$ containing $E$. Then, by the Riemann-Hurwitz formula, the connectivity of $A_1$ is greater than or equal to that of $A$. This implies that $\wh{\mathbb{C}}\sm E$ is disconnected, contradicting the condition that $E$ is full. So $f(E)$ is full. 
\vspace{2pt}

(2) Since the number of essential components in $\mc{E}$ is finite, there exists $N>0$ such that $\sigma_f^n(E)$ is non-essential for every $n\ge N$. Then $$f^k(\wh{\sigma_f^N(E)})=\wh{\sigma_f^{N+k}(E)},\ \forall\, k\geq0.$$

If $\sigma_f^{N}(E)$ is not full, then any component of $\wh{\sigma_f^N(E)}\sm\sigma_f^N(E)$ would be a wandering Fatou domain since $\wh{\sigma_f^n(E)}\cap\wh{\sigma_f^m(E)}=\emptyset$ for distinct $m,n\geq 0$. This contradicts Sullivan's eventually periodic theorem. Hence $\sigma_f^{n}(E)$ is full for every $n\geq N$ by statement (1).\vspace{2pt}

(3) It is obviously that $E$ is non-full since $\wh{E}\cap f^{-1}(U)\neq \emptyset$. 
Choose a Jordan disk $D\supset \wh{\sigma_f(E)}$ in an arbitrary small neighborhood of $\wh{\sigma_f(E)}$. Let $W$ be the unique component of $f^{-1}(D)$ containing $E$. Since $E$ is essential, $\wh{E}$ contains poles of $f$, but $W$ does not.  Thus $W$ is not simply connected. By Riemann-Hurwitz formula, $W$ must contain critical points of $f$. So $\wh{E}$ contains critical points since $D$ is an arbitrary small neighborhood of $\wh{\sigma_f(E)}$. \vspace{2pt}

(4) Choose a Jordan disk $D\supset \wh{\sigma_f(E)}$ in an arbitrary small neighborhood of $\wh{\sigma_f(E)}$. Let $W$ be the unique component of $f^{-1}(D)$ containing $E$. Since $E$ is non-full and $\partial D$ can be taken  arbitrary close to $\sigma_f(E)$, the domain $W$ can not be simply connected. It follows that $W$ contains critical points, and some bounded components of $\mathbb{C}\setminus \overline{W}$ contains poles of $f$. Then statement (4) follows as $D$ is arbitrary.
\end{proof}

We specify two  subsets of $\mc{E}$ that will be useful in the dicusson below. Set
\begin{equation}\label{eq:55}
\mc{E}_1=\{ \sigma_f^n(E) \,|\, n\geq 0, E \text{ is a preperiodic critical  element of  } \mc{E}\}\bigcup\{E_{x_*}\},
\end{equation}
where $E_{x_*}$ denotes the unique element of $\mc{E}$ containing the parabolic fixed point $x_*$, if $U$ is parabolic. Then $\mc{E}_1$ is a finite set and $\sigma_f(\mc{E}_1)= \mc{E}_1$. 

For any wandering element $E$ of $\mc{E}$, it follows from Lemma \ref{lem:basic}.(2) that there exists a minimal integer $n_E\geq 1$ such that $\sigma_f^{n_E}(E)$ is full. Set 
\begin{equation}\label{eq:66}
\mc{E}_2=\bigcup_{E}\{ E, \sigma_f(E),...,\sigma_f^{n_E-1}(E)\},
\end{equation}
where the union is taken over all  non-full, wandering and critical elements of $\mc{E}$.
\begin{lem}\label{lem:essential}
Every essential element of $\mc{E}$ belongs to $\mc{E}_1\cup\mc{E}_2$. 
\end{lem}
\begin{proof}
Let $E\in\mc{E}$ be any essential element. Then $E$ is a non-full and critical element of $\mc{E}$ by Lemma \ref{lem:basic}.(3). Hence $E\in\mc{E}_1$ if $E$ is pre-periodic, and $E\in\mc{E}_2$ otherwise. 
\end{proof}

\subsection{Construction of puzzle pieces from $U$.}
 In this section, we construct puzzle pieces for the rational map $f$ from the Fatou domain $U$, and  study their basic properties. \vspace{5pt}

\noindent -- {\it Attracting puzzle pieces}\vspace{5pt}

Assume first that $U$ is an attracting domain of $f$. Let $U_0\subset U$ be a Jordan domain containing the attracting fixed point in $U$ such that  $\pa U_0\cap P_f=\emptyset$ and $\ol{f(U_0)}\subset U_0$. For each $n\geq 1$, let $U_n$ be the component of $f^{-n}(U_0)$ containing $U_0$. We can choose a large $N$  such that the following properties hold:
\begin{enumerate}
	\item $U_N$ contains all critical points in $U$ and the point $\infty$;
	\item  the elements of $\mc{E}_1\cup \mc{E}_2 $ are contained in distinct components of $\wh{\mathbb{C}}\setminus U_N$;
	\item each component of $\cbar\setminus U_N$ contains at most one critical element of $\mc{E}$;
	\item If $E\in \mc{E}_1\cup\mc{E}_2$ and $D$ is a component of $\wh{\mathbb{C}}\setminus U_N$ containing $E$, then $D\setminus\wh{E}$ contains no critical values.
\end{enumerate}

For $n\ge 0$, the {\bf puzzle} of {\bf depth} $n$, denote by $\mc{P}_n$, is the collection of components of $\wh{\mathbb{C}}\setminus U_{N+n}$.
 It is straightforward to check the following properties.
\begin{itemize}
\item The set $\mc{P}_n$ contains finitely many elements.
\item Each element of $\mc{P}_n$, which is called a {\bf puzzle piece} of depth $n$, is a closed Jordan disk with the boundary contained in $U$. 
\item For any $E\in \mc{E}$, 
\begin{equation}\label{eq:33}
\bigcap_{n\geq 0}P_n(E)=\wh{E}.
\end{equation}
where $P_n(E)$ is the unique puzzle piece in $\mc{P}_n$ containing $E$ for $n\ge 0$.
\end{itemize}

According to the  above  conditions on $U_N$, we immediately obtain the following  mapping properties between puzzle pieces.
\begin{pro}\label{pro:mapping}
Let  $E$ be any element of $\mc{E}$, and $n\geq 0$ be any integer.  
\begin{enumerate}
\item If $E$ is essential, then $f (P_{n+1}(E)\sm\wh{E})=P_n(\sigma_f(E))\sm\wh{\sigma_f(E)}$.
\item If $P_{n+1}(E)$ contains no essential elements of $\mc{E}$, then $P_{n+1}(E)$ is a component of $f^{-1}(P_n(\sigma_f(E)))$.
\item If $P_{n+1}(E)$ contains no critical elements of $\mc{E}$, then $f:P_{n+1}(E)\to P_n(\sigma_f(E))$ is a homeomorphism.
\end{enumerate}
\end{pro}

\noindent -- {\it Parabolic puzzle pieces}\vspace{5pt}

Now suppose that $U$ is the parabolic domain of a parabolic fixed point $x_*\in\partial U$.   Let $U_0\subset U$ be a Jordan domain such that  
\[\text{$\pa U_0\cap \pa U=\{x_*\}, \pa U_0\cap P_f=\{x_*\}$ and $\ol{f(U_0)}\subset U_0\cup \{x_*\}$}.\]
For any $n\geq0$, let $U_n$ be the component of $f^{-n}(U_0)$ containing $U_0$. Then we can choose $N$ sufficiently large such that $U_N$ satisfies the same properties as  in the attracting cases. 

For $n\geq 0$, the (parabolic) {\bf puzzle} $\mc{P}_n$ of depth $n$ is the collection of components of $\mathbb{C}\sm U_{N+n}$, and its elements are called (parabolic) {\bf puzzle pieces} of level $n$.  For any $E\in\mc{E}$ and $n\geq0$, we still denote by $P_n(E)$  the unique puzzle piece of level $n$ containing $E$. The properties of parabolic puzzle pieces $P$ are as follows (compare with the attracting case).
\begin{itemize}
\item The set $\mc{P}_n$ contains finitely many elements.
\item Each element  $P\in \mc{P}_n$ is a full continuum. Moreover, $\partial P\cap J_f$ is a finite subset of $f^{-N-n}(x_*)$ and $\partial P\setminus J_f\subset U$.
\item For any $E\in \mc{E}$, 
\begin{equation}\label{eq:34}
\bigcap_{n\geq 0}P_n(E)=\wh{E}.
\end{equation}
\end{itemize}
Since $U_N$ satisfies the same properties in both attracting and parabolic cases, we have the following result.
\begin{pro}\label{pro:mapping2}
The conclusions in Proposition \ref{pro:mapping} still hold in the parabolic case.
\end{pro}

\subsection{Two reduction lemmas}
In this part, we present two lemmas  which together easily imply Theorem \ref{thm:main2}. For any subset $\mc{E}'$ of $\mc{E}$ and any $n\geq0$, denote by $\mc{P}_n(\mc{E'})$ the collection of level-$n$ puzzle pieces that contain elements of $\mc{E}'$. Note that $\mc{P}_n=\mc{P}_n(\mc{E})$. 

\begin{lem}\label{lem:key1}
Let $E$ be an element of $\mc{E}$. Suppose that there exists an increasing sequence $\{n_k\}_{k\geq1}$  such that $f^{n_k}(P_{n_k}(E))\in \mc{P}_0\sm \mc{P}_0(\mc{E}_1)$ for every $k\geq1$.  Then $E$ is a singleton, provided that each wandering exposed Julia component of $f$ containing critical points is non-recurrent.
\end{lem}

The proof relies on the following Shrinking Lemma, see \cite[Section 12]{lyubich1997laminations}.

\begin{lem}[Shrinking Lemma]\label{Shrinking}
Let $D_0\subset \wh{\mathbb{C}}$ be a Jordan disk which is not contained in any rotation domain of $f$, and $\{D_n\}_{n\geq 1}$ be a sequence of domains such that $D_n$ is a component of $f^{-1}(D_{n-1})$. If $\deg (f^n: D_n\to D_0)\leq M$ for every $n\geq 1$, then for any compact set $K\subset D_0$, diam$(K_n)\to 0$  as $n\to\infty$, where $K_n=f^{-n}(K)\cap D_n$.
\end{lem}

\begin{proof}[Proof of Lemma \ref{lem:key1}]
Suppose that every wandering exposed Julia component containing critical points is non-recurrent.
Then \vspace{3pt}

\emph{Claim 1. Each wandering critical element of  $\mc{E}$ is non-recurrent under $\sigma_f$.}\vspace{3pt}

To prove this claim, let $E'$ be a wandering critical element of $\mc{E}$. If $E'$ is of full-type, then it is a Julia component containing critical points. So $E'$ is non-recurrent by the assumption. If $E'$ is not full, by Lemma \ref{lem:basic} (2) and (4), the orbit of $E$ contains an element $E''$ of $\mc{E}$, such that the unique Julia component $J$ containing $E''$ contains critical points of $f$. Since $J$ is assumed to be non-recurrent. It follows that $E''$, and hence $E'$, is non-recurrent.  Then Claim 1 is proved.

By Claim 1, we can assume that, for each wandering critical element $E'$ of $\mc{E}$, its orbit $\{\sigma_f^n(E')\}_{n\geq1}$ is disjoint from $P_0(E')$. According to Propositions \ref{pro:mapping} and \ref{pro:mapping2}, if there exists an increasing sequence $\{n_k\}_{k\geq1}$  such that $f^{n_k}(P_{n_k}(E))\in \mc{P}_0\sm \mc{P}_0(\mc{E}_1)$ for every $k\geq1$, then 
\begin{itemize}
\item [(a)] for any $0\leq j\leq n_k-1$, $P_{n_k-j}(\sigma_f^j(E))$ contains no essential elements of $\mc{E}$; 
\item [(b)] for any $k\geq1$ and $0\leq j\leq n_k$, it holds that $f^j(P_{n_k}(E))=P_{n_k-j}(\sigma_f^j(E))$;
\item [(c)] for any $k\geq 1$, $P_{n_k}(E)$ is a component of $f^{-n_k}(P)$ for some $P\in\mc{P}_0\sm \mc{P}_0(\mc{E}_1)$. 
\end{itemize}

\emph{Claim 2. For any $k\geq1$, the finite orbit $\{f^j(P_{n_k}(E)),j=0,\ldots,n_k\}$ meets each critical element of $\mc{E}$ at most once.}\vspace{3pt}

To prove this claim, note first that for any $j=0,\ldots,n_k$, the puzzle piece $f^j(P_{n_k}(E))$ do not meet any preperiodic critical element of $\mc{E}$: otherwise $f^{n_k}(P_{n_k}(E))=P_0(\sigma_f^{n_k}(E))$ would belong to $\mc{P}_0(\mc{E}_1)$, a contradiction to statement (c) above.

Suppose on the contrary that the orbit $\{P_{n_k-j}(\sigma_f^j(E)),j=0,\ldots,n_k\}$ meet a wandering critical element $E'$ at least twice, for example, $P_{n_k-j}(\sigma_f^j(E))=P_{n_k-j}(E')$ for $j=j_1,j_2$ with $j_1>j_2\geq 0$. Then $\sigma_f^{j_1-j_2}(E')\in P_0(E')$. This contradicts the assumption that the orbit $\{\sigma_f^n(E')\}_{n\geq1}$ is disjoint from $P_0(E')$. So Claim 2 is proved. 
\vspace{5pt}

By Claim 2 and Propositions \ref{pro:mapping}, \ref{pro:mapping2}, we obtain that 
\begin{equation}\label{eq:44}
{\rm deg}(f^{n_k}: P_{n_k}(E)\to P_0(\sigma_f^{n_k}(E)))\leq d^{m_0}
\end{equation}
for any $k\geq1$, where $m_0$ denotes the number of wandering critical elements of $\mc{E}$. 

For any puzzle piece $P\in\mc{P}_0\sm\mc{P}_0(\mc{E}_1)$, since $P$ is a component of $\widehat{\mathbb C}\setminus U_{N}$, we can choose a Jordan disk $D_P\supset P$ such that $D_P\setminus P\subset U$. As $P$ does not contain the parabolic fixed point $x_*$, by shrinking $D_P$, we can assume that $D_P\setminus P$ is disjoint from $P_f$.  Thus, if $f^{n_k}(P_{n_k}(E))=P$, then  $${\rm deg}(f^{n_k}:D_P^{n_k}\to D_P)\leq d^{m_0}$$ by \eqref{eq:44}, where $D_P^{n_k}$ is the component of $f^{-n_k}(D_P)$ that contains $P_{n_k}(E)$. As a consequence, we can apply Shrinking Lemma to $(D_P,P)$ for all $P\in \mc{P}_0\sm\mc{P}_0(\mc{E}_1)$, and deduce that $E$ is a singleton.
\end{proof}

\begin{lem}\label{lem:key2}
Let $E$ be an element of $\mc{E}$ such that the orbit of $E$ does not accumulate to any element in $\mc{E}_2$.
Then there exists a sequence of integers $\{n_k\}$ satisfying the property of Lemma \ref{lem:key1} for $E$.
\end{lem}

\begin{proof}
\emph{Case 1. The orbit of $E$ does not accumulate to elements of $\mc{E}_1$}\vspace{3pt}
 
By enlarging $N$ if necessary,  we may assume that the orbit of $E$ is disjoint from the elements of $\mc{P}_0(\mc{E}_1\cup\mc{E}_2)$. Since $\mc{E}_1\cup\mc{E}_2$ contains all essential elements of $\mc{E}$ (by Lemma \ref{lem:essential}), it follows that the puzzle piece $P_n(\sigma_f^k(E))$ contains no essential elements of $\mc{E}$ for any $n,k\geq0$.
 By Propositions \ref{pro:mapping} and \ref{pro:mapping2}, we obtain that $f^n(P_n(E))\in\mc{P}_0\sm\mc{P}_0(\mc{E}_1)$ for any $n\geq1$. Then the sequence $\{n\}_{n\geq1}$ satisfies the condition of Lemma \ref{lem:key1}.\vspace{3pt}
 
 \emph{Case 2. The orbit of $E$ accumulates to elements of $\mc{E}_1$.}\vspace{3pt}
 
Set 
$$\mc{E}'_1=\sigma_f^{-1}(\mc{E}_1)\sm\mc{E}_1. $$
By enlarging $N$ if necessary, we may assume that  $\mc{P}_0(\mc{E}'_1)\cap \mc{P}_0(\mc{E}_1\cup \mc{E}_2)=\emptyset$, $P_0(E)\not\in \mc{P}_0(\mc{E}_1)$ and the orbit of $E$ is disjoint from  $ \mc{P}_0(\mc{E}_2)$.

Since the orbit of $E$ accumulates to $\mc{E}_1$, there exists a minimal integer $n_1\geq 1$  such that $\sigma_f^{n_1}(E)$ enters an element of $\mc{P}_{1}(\mc{E}_1)$. Due to the minimum of $n_1$,  the set $\sigma_f^{n_1-1}(E)$ is not contained in any element of $\mc{P}_1(\mc{E}_1)$. Then Propositions \ref{pro:mapping} and  \ref{pro:mapping2} imply that  $\sigma_f^{n_1-1}(E)$ lies in an element of $\mc{P}_1(\mc{E}_1')$. Hence 
\begin{equation}\label{eq:77}
P_0(\sigma_f^{n_1-1}(E))\not\in \mc{P}_0(\mc{E}_1).
\end{equation}

From \eqref{eq:77} and the $\sigma_f$-invariance of $\mc{E}_1$, we can deduce that
$$
P_1(\sigma_f^{n_1-2}(E))\notin \mc{P}_1(\mc{E}_1). 
$$
To see this, suppose on the contrary that $P_1(\sigma_f^{n_1-2}(E))=P_1(E')$ for some $E'\in \mc{E}_1$. Then $\sigma_f^{n_1-2}(E)\subset P_1(E')\sm\wh{E'}$, and thus $\sigma_f^{n_1-1}(E)\subset P_0(\sigma_f(E'))$ due to Proposition \ref{pro:mapping}.(1),(2) and Proposition \ref{pro:mapping2}. Since  $\sigma_f(E')\in  \mc{E}_1$,  it follows that $\sigma_f^{n_1-1}(E)$ is contained in an element of $\mc{P}_0(\mc{E}_1)$, a contradiction.   See Figure \ref{fig:key}. 

\begin{figure}[htbp]
\centering
\includegraphics[width=14cm]{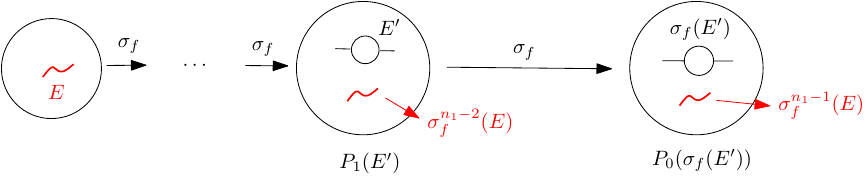}
\caption{}
\label{fig:key}
\end{figure}

Inductively using the above argument, we obtain that 
$$
P_i(\sigma_f^{n_1-(i+1)}(E))\notin \mc{P}_i(\mc{E}_1),\ i=1,2,\cdots,n_1-1.$$ 
Since the orbit of $E$ avoids $\mc{P}_0(\mc{E}_2)$, we also have that 
$$P_i(\sigma_f^{n_1-(i+1)}(E))\notin \mc{P}_i(\mc{E}_2),\ i=1,2,\cdots,n_1-1.$$
Note that the essential elements of $\mc{E}$ are contained in $\mc{E}_1\cup\mc{E}_2$. 
Then $P_i(\sigma_f^{n_1-(i+1)}(E))$ contains no essential elements of $\mc{E}$ for every $i=1,\ldots, n_1-1$. It follows from Propositions \ref{pro:mapping} and \ref{pro:mapping2} that $f^{n_1-1}(P_{n_1-1}(E))\in \mc{P}_0\sm\mc{P}_0(\mc{E}_1)$.

For any $k\geq 1$, there exists a minimal integer $n_k$ such that $\sigma_f^{n_k}(E)$ lies in an element  of $\mc{P}_{k}(\mc{E}_1)$.  By taking a subsequence, we may assume that $n_k$ is increasing. 
With the same argument as above, we can prove that  $f^{n_k-1}(P_{n_k-1}(E))\in \mc{P}_0\sm\mc{P}_0(\mc{E}_1)$. Thus, the sequence $\{n_k\}_{k\geq1}$ is as required.
\end{proof}

\subsection{Wandering exposed components are singletons}
\begin{proof}[Proof of Theorem\ref{thm:main2}]
Let $J$ be any wandering exposed Julia component of $f$. Without loss of generality, we can assume that $J$ contains a wandering component $E_*$ of $\partial U$, where $U$ is a fixed Fatou domain of $f$ and $\infty\in U$. Then $J\subset \wh{E}_*$. Hence it suffices to prove that $E_*$ is a singleton. 

We first show that $\mc{E}_2=\emptyset$. Suppose on the contrary that $\mc{E}_2\not=\emptyset$. By Claim 1 in the proof of Lemma \ref{lem:key1}, each critical component of $\partial U$ is non-recurrent. Hence, there exists an element $E$ of $\mc{E}_2$ whose orbit does not accumulate to any element of $\mc{E}_2$.  Applying Lemmas \ref{lem:key1} and \ref{lem:key2} to $E$ immediately implies that $E$ is a singleton, which contradicts the definition of $\mc{E}_2$.

Since $\mc{E}_2=\emptyset$, the condition of Lemma \ref{lem:key2} is automatically satisfied. Then there exists a sequence of integers $\{n_k\}_{k\geq1}$ satisfying the property of Lemma \ref{lem:key1} for $E_*$. It follows that $E_*$ is a singleton, and the proof is completed.  
\end{proof}

\section{Finiteness of periodic Julia components}
This section is devoted to prove Theorem \ref{thm:main1}. We first investigate the configuration of  periodic complex-type Julia components.
Then the theorem follows by a contradiction argument.

\subsection{Configuration of complex-type periodic Julia components}

Let $f$ be a rational map of degree $d\geq 2$. We may assume that $f(\infty)=\infty$. Let $V_f$ be the set of critical values of $f$. Recall that for a compact connected set $E\subset \mathbb{C}$, $\wh{E}$ denotes the union of $E$ and all bounded components of $\mathbb{C}\setminus E$.

Let $\mc{C}$ denote the collection of 
periodic complex-type Julia components. We specify the  important subsets $\mc{S},\mc{A}$ and $\mc{D}$ of $\mc{C}$ as follows: for any element $K\in\mc{C}$, 
\begin{itemize}
\item if $K\cap (V_f\cup \{\infty\})\not=\emptyset$  or  $\mathbb{C}\sm K$  has at least two bounded components  intersecting $V_f$, then  $K$ is an element of $\mc{S}$;

 \item  if $f^n(K)\not\in\mc{S}$ for every $n\geq0$ and $\mathbb{C}\sm K$ has exactly one bounded component intersecting $V_f$, then $K$ is an element of $\mc{A}$. The bounded component intersecting $V_f$ is denoted as $V(K)$;
 
  \item  if $f^n(K)\not\in\mc{S}$ for every $n\geq0$ and $\wh{K}\cap V_f=\emptyset$, then $K$ is an element of $\mc{D}$.  \end{itemize}
It is easy to check from the definition that 
$\mc{S}$ is a finite set and $f(\mc{A}\cup \mc{D})\subset(\mc{A}\cup \mc{D})$.

Denote by $p(K)$ the period of a periodic Julia component $K$. Note that for any $K\in\mc{A}$, there exists a smallest integer $n(K)\in[1,p(K)]$ such that $f^{n(K)}(K)\in\mc{A}$. Thus, we can  define the first return map $F: \mc{A}\to \mc{A}$ such that  $F(K):=f^{n(K)}(K)$ for any $K\in\mc{A}$. 
\begin{lem}\label{lem:covering}
For any $K\in\mc{A}$, there exists a unique bounded component $U(K)$ of $\mathbb{C}\sm K$, such that the following map is a covering:
$$
f^{n(K)}:\wh{K}\sm U(K)\to\wh{F(K)}\sm V(F(K))
$$ 
\end{lem}

\begin{proof}
Set $K':=f^{n(K)-1}(K)$. Then $f(K')=F(K)$. Since $\wh{F(K)}\sm V(F(K))$ is disjoint from $V_f$, the component of $f^{-1}(\wh{F(K)}\sm V(F(K)))$ containing $K'$ has exactly two complementary components. One of them is unbounded since $f(\infty)=\infty$. Let $W$ be the bounded one. Then $f:\, \wh{K'}\sm W\to\wh{F(K)}\sm V(F(K))$ is a covering.

Note that $f^{n(K)-1}: \wh{K}\to\wh{K'}$ is a homeomorphism, because  $f(K),\ldots,f^{n(K)-1}(K)\in\mc{D}$ by the definition of $n(K)$. Thus there exists a unique bounded component of $\cbar\sm K$, denoted by $U(K)$, such that $f^{n(K)-1}: \wh{K}\sm U(K)\to\wh{K'}\sm W$ is a homeomorphism. This implies that $f^{n(K)}:\wh{K}\sm U(K)\to\wh{F(K)}\sm V(F(K))$ is a covering. 
\end{proof}

There is an important subset $\mc{B}$ of $\mc{A}$ in the following argument, defined as
$$
\mc{B}:=\{K\in\mc{A}: U(K)\neq V(K)\}.
$$
Note  that for any $K\in\mc{B}$, there exists a smallest integer $m(K)\in[1,p(K)]$ such that $f^{m(K)}(K)\in\mc{B}$. This allows us to define the first return map $G:\mc{B}\to\mc{B}$, defined by $G(K):=f^{m(K)}(K)$ for every $K\in \mc{B}$. It follows from Lemma \ref{lem:covering} that
$$
f^{m(K)}:\wh{K}\sm U(K)\to\wh{G(K)}\sm V(G(K))
$$
is a covering. It is worth noting that $G=F^k$ on $K$ for some integer $k\ge 1$, and this number $k$ is dependent on $K$.  

We can define an  equivalence relation on $\mc{B}$ relative to $V_f$: for $K,K'\in\mc{B}$, define $K\sim K'$ if $V(K)\cap V_f=V(K')\cap V_f$, or equivalently, $\wh{K}\cap V_f=\wh{K'}\cap V_f$. Note that $\mc{B}$ has finitely many $\sim$ equivalence classes. 

There is a natural total ordering relation on each $\sim$ equivalence class of $\mc{B}$: for any $K,K'\in\mc{B}$ in the same $\sim$ equivalence class, define $K\prec K'$ if $\wh{K'}\subset \wh{K}$.  
The following lemma characterizes some key properties of this ordering relation.

\begin{lem}\label{lem:order}
Given any  $K\in\mc{B}$, the following statements hold.
\begin{enumerate}
\item If $K'\in\mc{B}$ and $K\prec K'$, then $n(K)\le m(K)<n(K')\le m(K')$.
\item The collection $\{K'\in\mc{B}:K'\prec K\}$ is a finite set.
\item If $G(K)\sim K$, then either $G(K)=K$ or $G(K)\prec K$.
\end{enumerate}
\end{lem}

\begin{proof}
(1) By definition, we immediately obtain  $n(K)\le m(K)$ and $n(K')\le m(K')$. So it remains to show $m(K)<n(K')$.

By definition again, there exists $k\geq1$ such that $G=F^k$ on $K$, and it holds that  $U(K)\neq V(K)$ and $U(F^i(K))=V(F^i(K))$ for $1\le i<k$. Set $m_0:=n(K)$ and $m_i:=n(K)+\cdots+n(F^i(K))$ for $1\le i<k$. Then $m(K)=m_{k-1}$. 

Since $V(K)\neq U(K)$, the image $f^{n(K)}(V(K))$ is a bounded component of $\cbar\sm F(K)$ disjoint from $V(F(K))$. If $k=1$, then $m(K)=n(K)$, and hence $f^{m(K)}(V(K))$ is a bounded component of $\cbar\sm F(K)$ disjoint from $V(F(K))$. If $k>1$, then $f^{n(K)}(V(K))\neq U(F(K))$. Thus $f^{m_1}(V(K))$ is a bounded component of $\cbar\sm F^2(K)$ disjoint from $V(F^2(K))$. Inductively, $f^{m(K)}(V(K))=f^{m_{k-1}}(V(K))$ is a bounded component of $\cbar\sm F^{k}(K)$ disjoint from $V(F^{k}(K))$ (see Figure \ref{fig:G3}). In summary, $f^{j}(V(K))$ is a bounded component of $\cbar\sm f^j(K)$ disjoint from $V_f$ for $1\le j\le m(K)$.

As $K'\subset V(K)$, it follows that $f^j(K')\subset f^j(V(K))$ for $1\le j\le m(K)$. Then $f^j(K')\in\mc{D}$ for $1\le j\le m(K)$, and hence $m(K)<n(K')$ by the definition of $n(K')$.

\begin{figure}[htbp]
\centering
\includegraphics[width=15cm]{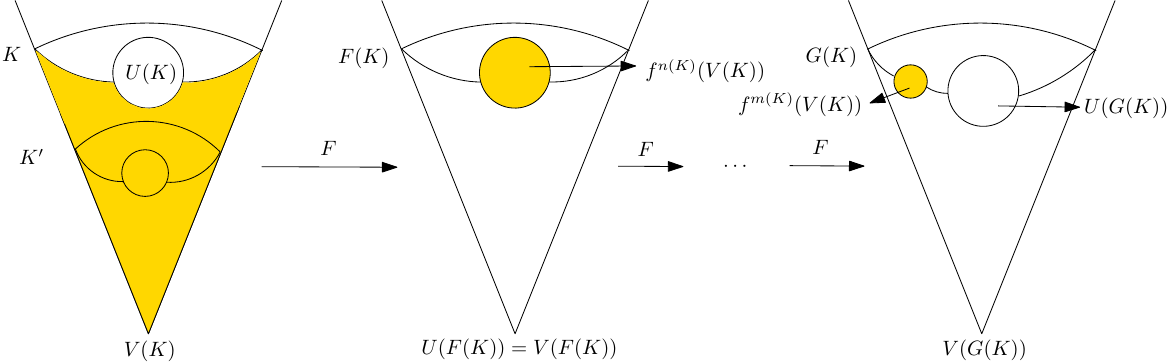}
\caption{}
\label{fig:G3}
\end{figure}
\vspace{3pt}

(2) This statement is a immediate consequence of statement (1).

(3) Assume by contradiction that $K\prec G(K)$. Then $V(K)\supset\wh{G(K)}$. It follows from (1) that $m(K)<m(G(K))$. Thus $f^{m(K)}(V(K))$ is a bounded component of $\cbar\sm G(K)$ and hence $f^{m(K)}(V(K))\subset V(K)$ (see Figure \ref{fig:G1}). Therefore, the sequence $\{f^{k\cdot m(K)}\}_{k\ge 1}$ forms a normal family, which implies  $V(K)\subset F_f$. It is impossible because $V(K)$ contains the Julia component $G(K)$.
\end{proof}
\begin{figure}[htbp]
\centering
\includegraphics[width=5.5cm]{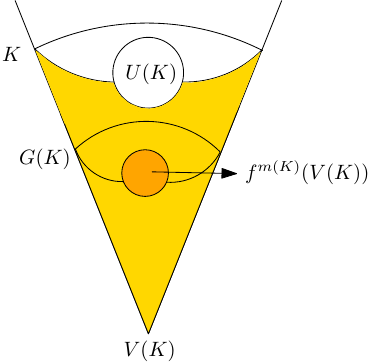}
\caption{}
\label{fig:G1}
\end{figure}

\subsection{Proof of Theorem \ref{thm:main1}}
The proof of Theorem \ref{thm:main1} goes by contradiction. Suppose that $f$ has infinite periodic complex-type Julia components. Then $\#\mc{C}=\infty$.  A key step in the proof is to check $\#\mc{B}=\infty$. This result follows easily from the following three lemmas. 

\begin{lem}\label{2.1}
If $\#\mc{C}=\infty$, then the periods $p(K)$ for all $K\in\mc{C}$ goes to infinity.
\end{lem}

\begin{proof}
Fix any $m\geq1$. Let $\mc{C}_m$ denote the collection of elements of $\mc{C}$ with period exactly $m$. By Theorem A, the action of $f^{m}$ on any element $K$ of $\mc{C}_m$ is quasi-conformally conjugate to a rational map on its Julia set. Together with Corollary 12.7 in \cite{milnor2006dynamics}, it follows that  $K$ contains periodic points of $f$ with period $m$. Since $f$ has finitely many periodic points of period $m$, we have  $\#\mc{C}_m<\infty$. Then $\#\mc{C}=\infty$ implies $\#\mc{C}_m\to\infty$ as $m\to\infty$. 
\end{proof}

 Let  $p_0$ be the maximum of the periods among all elements in $\mc{S}$, and let $p_1$ be the maximum of the periods among all periodic Fatou domains of $f$. 

\begin{lem}\label{lem:A}
For any $K\in\mc{C}$ with $p(K)>p_0$, there exists an integer $n_0\in[1,p(K)]$ such that $f^{n_0}(K)\in\mc{A}$.
\end{lem}

\begin{proof}
Since $p(K)>p_0$, the orbit of $K$ is disjoint from the set $\mc{S}$. Assume on the contrary that the  lemma is false. Then $K,\ldots, f^{p(K)-1}(K)\in\mc{D}$.  It follows that $f^{p(K)}:K\rightarrow K$ is a homeomorphism, because $\wh{f^j(K)}\cap V_f=\emptyset$ for $j=0,\ldots,p(K)-1$. This contradicts Theorem A.
\end{proof}

\begin{lem}\label{lem:UnV}
For any $K\in\mc{A}$ with $p(K)> p_1$, there exists an integer $k\geq 1$ such that $f^k(K)\in \mc{B}$.
\end{lem}

\begin{proof}
Assume by contradiction that the orbit of $K$ never enters $\mc{B}$. Then for any $k\geq1$, we have that $U(f^n(K))=V(f^n(K))$. As a result, for any bounded component $\Omega$ of $\cbar\sm K$ with $\Omega\neq U(K)$,  the $n$-th iteration $f^n(\Omega)$ is a bounded component of $\cbar\sm f^n(K)$ for all $n\ge 1$. Therefore, $\{f^n(\Omega)\}_{n\ge 0}$ forms a normal family, and then $\Omega$ is a Fatou domain. 

Since $p(K)>p_1$, $\Omega$ is a wandering Fatou domain, a contradiction. Hence $\mathbb{C}\setminus K$ has the unique bounded component $U(K)$, which contradicts $K\in\mc{C}$. 
\end{proof}

\begin{pro}\label{pro:finite}
If $\#\mc{C}=\infty$, then $\#\mc{B}=\infty$.
\end{pro}
\begin{proof}
Note that if $K\in\mc{C}$ has period larger than $p_0$, then $K\in \mc{A}\cup\mc{D}$. According to Lemma \ref{2.1}, there are infinitely many elements of $\mc{C}$ with periods larger than $p_0$. Then $\#(\mc{A}\cup\mc{D})=\infty$. Lemma \ref{lem:A} implies further that $\#\mc{A}=\infty$. Finally, by applying Lemma \ref{lem:UnV}, we can deduce $\#\mc{B}=\infty$ from the fact that $\#\mc{A}=\infty$.
\end{proof}

\begin{proof}[Complete the proof of Theorem \ref{thm:main1}]
Since  $\#\mc{B}=\infty$, there exists $\sim$ equivalent classes with infinitely many elements.  Denote by $\mc{B}_1,\ldots,\mc{B}_q$ all such equivalence classes.
For any $i\in\{1,\ldots,q\}$, set  
\begin{equation}\label{eq:222}
E_{i}:= \bigcap_{K\in\mc{B}_i}V(K).
\end{equation}

\emph{Claim 1. Each $E_{i}$ is a wandering Julia component containing critical values.}
\begin{proof}[Proof of Claim 1]
Notice that if $K\prec K'\in \mc{B}_i$, then $V(K')\Subset V(K)$ and $V(K')\cap V_f=V(K)\cap V_f\not=\emptyset$. So $E_{i}$ is a full, connected and compact set intersecting $V_f$, with $\partial E_i\subset J_f$. Let $J_i$ denote the Julia component containing $\partial E_i$. Clearly, $J_i\subset V(K)$ for all $K\in \mc{B}_i$. This implies $J_i=\partial E_i$. 

For any $K\in \mc{B}_i$, since $U(K)\not=V(K)$,  if follows from Lemma \ref{lem:covering} that $f^{n(K)-1}$ is injective on $V(K)$. Due to Lemma \ref{lem:order}.(1),  the integers $n(K), K\in \mc{B}_i$ are unbounded. So $f^n$ is injective on $E_{i}$ for all $n\ge 0$. It then follows from Theorem A that $E_i$ is wandering. Hence $E_i=\partial E_i$ as there are no wandering Fatou domains. 
\end{proof}

Let  $\{K_k\}_{k\geq 1}$ be an infinite sequence in $\mc{B}_1$ such that $K_k\prec K_{k+1}$ and $p(K_k)<p(K_{k+1})$ (by Lemma \ref{2.1}) for every $k\geq 1$. 
Since $\mc{B}$ has finitely many $\sim$ equivalence classes, by taking subsequences, we may assume that 
$\{G(K_k)\}_{k\ge 1}$ belongs to a common $\sim$ equivalence class, denoted by $\mc{B}_2$. According to Lemma \ref{lem:order}.(2), we may further assume that $G(K_k) \prec G(K_{k+1})$ for each $k\geq1$. Obviously, $E_1=\bigcup_{k\geq1}V(K_k)$ and $E_2=\bigcap_{k\geq1}V(G(K_k))$. \vspace{5pt}

\emph{Claim 2. The orbit of $E_1$ accumulates to $E_2$.}
\begin{proof}[Proof of Claim 2]
By the defintion of $G:\mc{B}\to \mc{B}$, for any $k\geq1$, we have that $G(K_k)=f^{m(K_k)}(K_k)$ and $f^{m(K_k)}(V(K_k))$ is a bounded component of $\cbar\sm G(K_k)$ distinct from $V(G(K_k))$. Since $E_{1}\subset V(K_k)$ for all $k\ge 1$, it follows that  $f^{m(K_k)}(E_1)$ is contained in a bounded component of $\cbar\sm G(K_k)$ distinct from $V(G(K_k))$, for every $k\geq1$. Since  $G(K_k)\prec G(K_{k+1})$, the sequence  $f^{m(K_k)}(E_1), k\geq 1$ are pairwise disjoint. On the other hand, as $E_{2}= \bigcap_{k\ge 1}V(G(K_k))$, we obtain that $f^k(E_1)\cap E_2=\emptyset$ for all $k\geq 1$, but the orbit of $E_1$ accumulates $E_2$.  See Figure \ref{fig:recurrent}. \end{proof}

\begin{figure}[htbp]
\centering
\includegraphics[width=12cm]{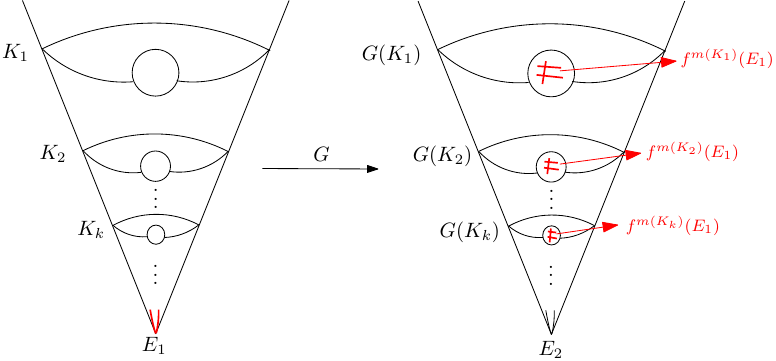}
\caption{}
\label{fig:recurrent}
\end{figure}

By inductively using the argument above, we obtain a sequence of sets $\{E_j\}$ defined in \eqref{eq:222} such that the orbit of $E_j$ accumulates to $E_{j+1}$. Since there are only $q$ distinct $E_j$, there must be an $j_0$ such that the orbits of $E_{j_0}$  accumulates to itself. By Claim 1,  such $E_{j_0}$ is a recurrent Julia component containing critical values,   
a contradiction. 
Then we  complete the proof of Theorem \ref{thm:main1}. 
\end{proof}

\bibliography{wandering}{}

\begin{thebibliography}{10}

\bibitem{beardon1991iteration}
A.~F. Beardon.
\newblock {\em {Iteration of rational functions: Complex analytic dynamical
  systems}}, volume 132 of {\em GTM}.
\newblock Springer, 1991.

\bibitem{beardon1991the}
A.~F. Beardon.
\newblock {The components of a Julia set}.
\newblock {\em Annales Academi{$\ae$} Scientiarum Fenicae, Series A},
  16:173--177, 1991.

\bibitem{branner1992iteration}
B.~Branner and J.~H. Hubbard.
\newblock {The iteration of cubic polynomials. {II}. Patterns and
  parapatterns}.
\newblock {\em Acta Math.}, 169(3-4):229--325, 1992.

\bibitem{cui2011topology}
G.~Cui, W.~Peng, and L.~Tan.
\newblock {On the topology of wandering Julia components}.
\newblock {\em Discrete Contin. Dyn. Syst.}, 29(3):929--952, 2011.

\bibitem{cui2024wandering}
G.~Cui, W.~Peng, and L.~Yang.
\newblock Wandering {J}ulia components of cubic rational maps.
\newblock {\em Math. Z.}, 308(4):Paper No. 69, 10, 2024.

\bibitem{kozlovski2009local}
O.~Kozlovski and S.~van Strien.
\newblock Local connectivity and quasi-conformal rigidity of non-renormalizable
  polynomials.
\newblock {\em Proc. Lond. Math. Soc. (3)}, 99(2):275--296, 2009.

\bibitem{lyubich1997laminations}
M.~Lyubich and Y.~Minsky.
\newblock {Laminations in holomorphic dynamics}.
\newblock {\em Journal of Differential Geometry}, 47(1):17--94, 1997.

\bibitem{mcmullen1988automorphisms}
C.~T. McMullen.
\newblock {Automorphisms of rational maps}.
\newblock In {\em {Holomorphic functions and moduli, Vol. I (Berkeley, CA,
  1986)}}, volume~10 of {\em Math. Sci. Res. Inst. Publ.}, pages 31--60.
  Springer, New York, 1988.

\bibitem{milnor2006dynamics}
J.~W. Milnor.
\newblock {\em {Dynamics in one complex variable. Third Edition}}, volume 160
  of {\em AMS}.
\newblock Princeton University Press, 2006.

\bibitem{pilgrim2000rational}
K.~M. Pilgrim and L.~Tan.
\newblock {Rational maps with disconnected Julia set}.
\newblock {\em Ast\'erique}, 261:349--384, 2000.

\bibitem{qiu2009proof}
W.~Qiu and Y.~Yin.
\newblock {Proof of the Branner-Hubbard conjecture of Cantor Julia sets}.
\newblock {\em Science in China Series A: Mathematics}, 52:45--65, 2009.

\end{thebibliography}
\bibliographystyle{abbrv}
\end{document}